\documentclass{article}
\usepackage{amsmath}
\usepackage{amsfonts}
\usepackage{amsthm}
\usepackage{amssymb}
\usepackage{color}
\usepackage{graphicx}
\usepackage{float}
\usepackage{subfigure}
\usepackage[font=small]{caption}
\usepackage{tikz}
\usepackage{enumitem}

\newtheorem{theorem}{Theorem}[section]
\newtheorem{lemma}[theorem]{Lemma}

\newtheorem{observation}[theorem]{Observation}

\newtheorem{example}[theorem]{Example}

\theoremstyle{definition}
\newtheorem{definition}[theorem]{Definition}
\theoremstyle{remark}
\newtheorem{remark}[theorem]{Remark}

\newcommand\remove[1]{}

\def\mh{\mathcal{H}}

\def\f2{\mathbb{F}_2}

\def\lip{\hskip0.02cm{\rm Lip}\hskip0.01cm}

\newcommand{\1}{\mathbf{1}}

\newcommand{\tc}{{\rm TC}\hskip0.02cm}

\newcommand{\tp}{{\rm TP}\hskip0.02cm}

\begin{document}

\title{Isometric copies of $\ell_\infty^n$ and $\ell_1^n$ in transportation cost spaces on finite metric spaces}

\author{Seychelle~S.~Khan, Mutasim~Mim,  and Mikhail~I.~Ostrovskii}

\date{\today}
\maketitle

\hfill{\sl To Victor Lomonosov}
\bigskip

\noindent{\bf Abstract.} Main results: (a) If a metric space
contains $2n$ elements, the transportation cost space on it
contains a $1$-complemented isometric copy of $\ell_1^n$. (b) An
example of a finite metric space whose transportation cost space
contains an isometric copy of $\ell_\infty^4$. Transportation cost
spaces are also known as Arens-Eells, Lipschitz-free, or
Wasserstein $1$ spaces.
\medskip

\medskip

\medskip

\begin{large}


\section{Introduction}

The introduced below notions go back at least to Kantorovich and
Gavurin \cite{KG49}. We use the terminology and notation of
\cite{OO19}. History of the notions introduced below as well as
related terminology (Arens-Eells space, earth mover distance,
Kantorovich-Rubinstein distance, Lipschitz-free space, Wasserstein
distance) is discussed in \cite[Section 1.6]{OO19} and references
therein.

\begin{definition}
Let $(M,d)$ be a metric space. Consider a real-valued finitely
supported function $f$ on $M$ with a zero sum, that is,

\begin{equation}\label{E:Sum0}\sum_{v\in M}f(v)=0.\end{equation}

A natural and important interpretation of such a function is the
following: $f(v)>0$ means that $f(v)$ units of a certain product
are produced or stored at point $v$; $f(v)<0$ means that $(-f(v))$
units of the same product are needed at $v$.  The number of units
can be any real number.  With this in mind, $f$ may be regarded as
a {\it transportation problem}. For this reason, we denote the
vector space of all real-valued functions finitely supported on
$M$ with a zero sum by $\tp(M)$, where $\tp$ stands for {\it
transportation problems}.

One of the standard norms on the vector space $\tp(M)$ is related
to the {\it transportation cost} and is defined in the following
way.\medskip

A {\it transportation plan} is a plan of the following type: we
intend to deliver

\begin{itemize}

\item $a_1$ units of the product from $x_1$ to $y_1$,

\item $a_2$ units of the product from $x_2$ to $y_2$,

\item \dots

\item $a_n$ units of the product from $x_n$ to $y_n$,

\end{itemize}
where $a_1,\dots,a_n$ are nonnegative real numbers, and
$x_1,\dots,x_n,y_1,\dots,y_n$ are elements of $M$, which do not
have to be distinct.\medskip

This transportation  plan is said to  {\it solve the
transportation problem} $f$ if
\begin{equation}\label{E:TranspPlan}f=a_1(\1_{x_1}-\1_{y_1})+a_2(\1_{x_2}-\1_{y_2})+\dots+
a_n(\1_{x_n}-\1_{y_n}),\end{equation} where $\1_u(x)$ for $u\in M$
is the {\it indicator function} defined as:
\[\1_u(x)=\begin{cases} 1 &\hbox{ if }x=u,\\ 0 &\hbox{ if }x\ne u.
\end{cases} \]

 The {\it cost} of  transportation plan
\eqref{E:TranspPlan} is
 defined as $\sum_{i=1}^n a_id(x_i,y_i)$. We introduce the
{\it transportation cost norm} (or just {\it transportation cost})
$\|f\|_{\tc}$ of a transportation problem $f$ as the minimal cost
of transportation plans solving $f$. It is easy to see that the
transportation plan of the minimum cost exists. We introduce the
{\it transportation cost space} $\tc(M)$ on $M$ as the completion
of $\tp(M)$ with respect to the norm $\|\cdot\|_\tc$.
\end{definition}

It is worth mentioning that the norm of an element in $\tc(M)$ can
be computed using Linear Programming, see \cite{MG07} and
\cite{Sch86}, see also related historical comments in
\cite[pp.~221--223]{Sch86}.
\medskip

Arens and Eells \cite{AE56} observed that if we pick a  {\it base
point} $O$ in the space $M$, then the {\it canonical embedding}\,
of $M$ into $(\tp(M),\|\cdot\|_\tc)$ given by the formula:
\begin{equation}\label{E:EmbBasePt} v\mapsto \1_v-\1_O\end{equation}
is an isometric embedding. This observation can be easily derived
from the following characterization of optimal transportation
plans.
\medskip

Let $0\le C<\infty$. A real-valued  function $l$ on a metric space
$(M,d)$ is called $C$-{\it Lipschitz} if
\[\forall x,y\in M\quad |l(x)-l(y)|\le Cd(x,y).\]
The {\it Lipschitz constant} of a function $l$ on a metric space
containing at least two points is defined as
\[\lip(l)=\max_{x,y\in M,~~ x\ne y}\frac{|l(x)-l(y)|}{d(x,y)}.\]

\begin{theorem}[\cite{KG49}]\label{T:WitnLip} A plan
\begin{equation}f=a_1(\1_{x_1}-\1_{y_1})+a_2(\1_{x_2}-\1_{y_2})+\dots+
a_n(\1_{x_n}-\1_{y_n})\end{equation} is optimal if and only if
there exist a $1$-Lipschitz real-valued function $l$ on $M$ such
that \begin{equation}\label{E:Kan}
l(x_i)-l(y_i)=d(x_i,y_i)\end{equation} for all pairs $x_i,y_i$ for
which $a_i>0$.
\end{theorem}

The mentioned above observation of Arens and Eells makes
transportation cost spaces an important object in the theory of
metric embeddings, see \cite[Chapter 10]{Ost13} and \cite[Section
1.4]{OO19}. This theory makes it very important to study the
conditions of isometric embeddability of spaces $\ell_\infty^n$
into $\tc(M)$.

Problems on isometric embeddability of spaces  $\ell_1^n$ and
$\ell_\infty^n$ into $\tc(M)$ are also motivated by the following
definitions, the first of which goes back to Kantorovich and
Gavurin \cite{KG49}.

\begin{definition} Let $f_1,\dots,f_n$ be nonzero transportation problems in $\tp(M)$ and $x_1,\dots,x_n$ be their
normalizations, that is, $x_i=f_i/\|f_i\|_\tc$.
\smallskip

We say that transportation problems $f_1,\dots,f_n$ are {\it
completely unrelated}, if \[\left\|\sum_{i=1}^n
a_ix_i\right\|_\tc=\sum_{i=1}^n |a_i|\] for every collection
$\{a_i\}_{i=1}^n$ of real numbers.

We say that transportation problems $f_1,\dots,f_n$ are {\it
completely intertwined}, if \[\left\|\sum_{i=1}^n
a_ix_i\right\|_\tc=\max_{1\le i\le n} |a_i|\] for every collection
$\{a_i\}_{i=1}^n$ of real numbers.
\end{definition}

\begin{remark} The notion of completely unrelated problems has a
natural meaning in applications: we cannot decrease the total cost
by combining the transportation plans for a set of completely
unrelated transportation problems.

The notion of completely intertwined problems describes the very
unusual situation: we have several transportation problems
$\{x_i\}_{i=1}^n$ such that each of them has cost $1$ and the sum
$\sum_{i=1}^n\theta_ix_i$ (of $n$ summands with cost $1$ each) has
cost $1$ for every collection $\theta_i=\pm1$.

It is clear that problems are completely unrelated if and only if
their normalizations are isometrically equivalent to the unit
vector basis of $\ell_1^n$ and problems are completely intertwined
if and only if their normalizations are isometrically equivalent
to the unit vector basis of $\ell_\infty^n$.
\end{remark}

The main goal of this paper is to study embeddability of
$\ell_1^n$ and $\ell_\infty^n$ into $\tc(M)$ for finite metric
spaces $M$. The following theorem is our main result.

\begin{theorem}\label{T:2npts} If a metric space $M$ contains $2n$ elements, then $\tc(M)$ contains a
$1$-complemented subspace isometric to $\ell_1^n$. If the space
$M$ is such that triangle inequalities for all distinct triples in
$M$ are strict, then $\tc(M)$ does not contain a subspace
isometric to $\ell_1^{n+1}$.
\end{theorem}

\begin{remark} It can be easily seen from the proof that in the
case where a finite metric space $M$ contains more than $2n$
elements, the space $\tc(M)$ also contains a $1$-complemented
subspace isometric to $\ell_1^n$. This is not completely obvious
only if $|M|$ is odd. In this case we add to $M$ one point in an
arbitrary way, apply Theorem \ref{T:2npts}, and then observe that
all elements of standard basis of the constructed space, except
one, are contained in $\tc(M)$.
\end{remark}

Theorem \ref{T:2npts} solves \cite[Problem 3.3]{DKO18+} by
strengthening \cite[Theorem 3.1]{DKO18+} which states that for $M$
with $2n$ elements the space $\tc(M)$ contains a $2$-complemented
subspace $2$-isomorphic to $\ell_1^n$.
\medskip

Problems of isometric embeddability of $\ell_1$ into $\tc(M)$ on
infinite metric spaces $M$ were considered in \cite{CJ17,OO19}.
\medskip

The existing knowledge on embeddability of $\ell_\infty^n$ is very
limited. The most important sources in this direction are
\cite{Bou86} and \cite{GK03}. In Section \ref{S:ellinftyn} we
present a special case of one of the results of \cite{GK03} in the
form which, in our opinion, helps to understand the phenomenon.
Bourgain \cite{Bou86} proved (see also a presentation in
\cite[Section 10.4]{Ost13}) that $\tc(\ell_1)$ contains almost
isometric copies of $\ell_\infty^n$ for all $n$.\medskip

Our contribution to the case of $\ell_\infty^n$ (Section
\ref{S:ellinftyn}) consists in examples of relatively small finite
metric spaces $M_3$ and $M_4$ such that $\tc(M_3)$ and $\tc(M_4)$,
respectively, contain $\ell_\infty^3$ and $\ell_\infty^4$
isometrically. The reason for our interest to $M_3$ is that it is
smaller than $M_4$. We do not know whether such finite metric
spaces can be constructed for $\ell_\infty^n$ with $n\ge 5$.
\medskip

 In this connection
it is natural to recall the well-known fact that the spaces
$\ell_1^2$ and $\ell_\infty^2$ are isometric. It is easy to see
that the standard proof of this can be stated as:

\begin{observation}\label{O:Dim2} The transportation problems $f_1$ and $f_2$
are completely unrelated if and only if the transportation
problems $g_1=\frac12(f_1+f_2)$ and $g_2=\frac12(f_1-f_2)$ are
completely intertwined.
\end{observation}

\section{Proof of Theorem \ref{T:2npts}}\label{S:ell1_n}

We use terminology of \cite{Die17}. Consider the metric space $M$
as a weighted complete graph with $2n$ elements, we denote it also
$G=(V(G),E(G))$, the weight of an edge is the distance between its
ends. We consider matchings containing $n$ edges in this graph,
such matchings are called {\it perfect matchings} or {\it
$1$-factors}. We pick among all perfect matchings a matching of
minimum weight (the {\it weight} of a matching is defined as the
sum of weights of its edges). Let $e_1=u_1v_1, \dots, e_n=u_nv_n$
be a perfect matching of minimum weight. We claim that the
transportation problems
$f_1=\1_{u_1}-\1_{v_1},\dots,f_n=\1_{u_n}-\1_{v_n}$ are completely
unrelated.

We need to show that for any set $\{a_i\}_{i=1}^n$ of real numbers
we have
\[\left\|\sum_{i=1}^n
a_i\left(\1_{u_i}-\1_{v_i}\right)\right\|_\tc=\sum_{i=1}^n|a_i|d(u_i,v_i).\]
Assume for simplicity that all $a_i$ are positive (all other cases
can be done similarly, we can just interchange $u_i$ and $v_i$ for
those $i$ for which $a_i<0$).

The inequality
\[\left\|\sum_{i=1}^n
a_i\left(\1_{u_i}-\1_{v_i}\right)\right\|_\tc\le\sum_{i=1}^n|a_i|d(u_i,v_i)\]
is obvious. To prove the inverse inequality, assume the contrary,
that is,
\[\left\|\sum_{i=1}^n
a_i\left(\1_{u_i}-\1_{v_i}\right)\right\|_\tc<\sum_{i=1}^n|a_i|d(u_i,v_i).\]
In such a case there exist transportation plans for
$f=\sum_{i=1}^n a_i\left(\1_{u_i}-\1_{v_i}\right)$ with lower
costs than the straightforward plan (by the {\it straightforward
plan} we mean the plan in which $a_i$ units are moved from $u_i$
to $v_i$ for each $i=1,\dots,n$). Let
\begin{equation}\label{E:Opt}\sum_{j=1}^mb_j(\1_{x_j}-\1_{y_j}),\end{equation} where $b_j>0$ for
$j=1,\dots,m$, be an optimal plan for $f$, that is, a plan
satisfying
\[||f||_\tc=\sum_{j=1}^mb_jd(x_j,y_j).\]
It is known (see \cite[Proposition 3.16]{Wea18}) that such plans
exist and that there exists an optimal plan satisfying the
following condition:
\[{\bf (A)}\quad \hbox{Each }x_j\hbox{ is one of
}\{u_i\}_{i=1}^n\hbox{ and each }y_j\hbox{ is one of
}\{v_i\}_{i=1}^n.\]

Since plan \eqref{E:Opt} is different from the straightforward
plan and satisfies condition {\bf (A)}, there are
$n(0),n(1)\in\{1,\dots,n\}$, $n(0)\ne n(1)$ such that some amount
$c_0>0$ of the product is moved according to \eqref{E:Opt} from
$u_{n(0)}$ to $v_{n(1)}$. Then, there exist
$n(2)\in\{1,\dots,n\}$, $n(2)\ne n(1)$ such that some amount
$c_1>0$ of the product is moved according to \eqref{E:Opt} from
$u_{n(1)}$ to $v_{n(2)}$. We continue in an obvious way. Since we
consider finite sets, there is $k< n$ such that
$n(k+1)\in\{n(0),n(1),\dots,n(k)\}$. Without loss of generality we
may assume (changing the notation if necessary) that
$n(k+1)=n(0)$.

Let $c=\min_{0\le i\le k}c_i$. Then $c>0$ and part of the
transportation done according to the plan \eqref{E:Opt} is: $c$
units of the product are moved
\begin{itemize}
\item From $u_{n(0)}$ to $v_{n(1)}$, \item From $u_{n(1)}$ to
$v_{n(2)}$, \item \dots,\item From $u_{n(k)}$ to $v_{n(0)}$.
\end{itemize}

It is clear that if we modify this part of the plan to: $c$ units
of the product are moved
\begin{itemize}
\item From $u_{n(0)}$ to $v_{n(0)}$, \item From $u_{n(1)}$ to
$v_{n(1)}$, \item \dots,\item From $u_{n(k)}$ to $v_{n(k)}$,
\end{itemize}
we get another transportation plan for $f$.
\medskip

To clarify the main idea of the proof, first we consider the case
where $\{u_iv_i\}_{i=1}^n$ is a unique perfect matching of the
minimum weight, that is, all other perfect matchings have strictly
larger weights.

In this case we show that the cost of the modified (two paragraphs
above) transportation plan is strictly smaller than the cost of
\eqref{E:Opt}, and get a contradiction with the assumption that
\eqref{E:Opt} is an optimal plan.

To show this it suffices to prove that
\begin{equation}\label{E:(*)}
\sum_{i=0}^kcd(u_{n(i)},v_{n(i)})<\sum_{i=0}^kcd(u_{n(i)},v_{n(i+1)}),
\end{equation}
recall that $n(k+1)=n(0)$. Inequality \eqref{E:(*)} is an
immediate consequence of the assumption that the perfect matching
$\{u_iv_i\}_{i=1}^n$ has a strictly smaller weight than the
perfect matching
\[\{u_{n(i)}v_{n(i+1)}\}_{i=1}^k\bigcup\{u_iv_i\}_{i\in
R},\quad\hbox{ where }R=\{1,\dots,n\}\backslash
\{n(0),n(1),\dots,n(k)\},\] so the proof is completed under the
additional assumption of the uniqueness of the minimum weight
perfect matching.

Let us turn to the general case. In this case we can claim only a
non-strict inequality in \eqref{E:(*)}. This does not lead to an
immediate contradiction, but we can finish the argument in the
following way. Since $c>0$, the non-strict version of
\eqref{E:(*)} proves the following lemma.

\begin{lemma}\label{L:More} If an optimal transportation plan for $f$ satisfies {\bf (A)} and does not coincide with
the straightforward plan, then we can construct another optimal
transportation plan satisfying {\bf (A)} in which the total amount
of the product which is moved as in the straightforward plan, that
is, from $u_i$ to $v_i$ is strictly larger.
\end{lemma}

With this lemma we can complete the proof in the general case as
follows. Consider optimal transportation plans for $f$ satisfying
the condition {\bf (A)}. Such plans can be regarded as $n\times n$
matrices with nonnegative entries in which the entry $s_{i,j}$ is
the amount of the product which is to be moved from $u_i$ to
$v_j$. It is clear that the set of such optimal plans is closed in
any usual topology on the set of matrices. If it contains the
straightforward plan, we are done. If it does not, we get a
contradiction as follows. It is easy to check that among all
optimal plans satisfying condition {\bf (A)} there is a plan for
which the sum $\sum_{i=1}^n s_{i,i}$ is the maximal possible. If
this plan does not coincide with the straightforward, then, by
Lemma \ref{L:More}, there is an optimal plan satisfying {\bf (A)}
for which the sum $\sum_{i=1}^ns_{i,i}$ is larger, contrary to the
maximality assumption. This contradiction proves the existence in
$\tc(M)$ of the subspace isometric to $\ell_1^n$.
\medskip

Now, assume that $M$ is such that all triangle inequalities in $M$
are strict. Let $f_1,\dots, f_k$  be completely unrelated
transportation problems on $M$.

\begin{lemma} The functions $f_i$ have disjoint supports.
\end{lemma}

This lemma is essentially known \cite[Lemma 3.3]{OO19}, for
convenience of the reader we provide a proof.

\begin{proof} Assume the contrary, let $v\in M$ be in the supports of both $f_i$ and
$f_j$, $i\ne j$. Without loss of generality we assume that
$f_i(v)>0$  and $f_j(v)<0$, changing signs of $f_i$ and $f_j$ if
needed (the change of signs does not affect complete
unrelatedness).

To get a contradiction it suffices to show that
$\|f_i+f_j\|_\tc<\|f_i\|_\tc+\|f_j\|_\tc$. This can be done in the
following way. In an optimal plan for $f_i$ some amount of units,
denote it $\alpha>0$, is moved from $v$ to some $u\in M$. In an
optimal plan for $f_j$ some amount of units, denote it $\beta>0$,
is moved to $v$ from some $w\in M$ ($w$ can be the same as $u$).

Let $\gamma=\min\{\alpha,\beta\}$. Now we combine the optimal
plans for $f_i$ and $f_j$ with the following exception: we move
$\gamma$ units of the product directly from $w$ to $u$. Since, by
our assumption, $d(w,u)<d(w,v)+d(v,u)$, the cost of the obtained
plan is $<\|f_i\|_\tc+\|f_j\|_\tc$.
\end{proof}

Finally, since support of each function $f_i$ contains at least
two points, we get that $k\le n$. This proves the last statement
of Theorem \ref{T:2npts}.
\medskip

It remains to show that there is a projection of norm $1$ onto the
subspace spanned by $\{f_i\}_{i=1}^n$. We show that a linear
operator $P$ is a norm-$1$ projection onto the subspace spanned by
$\{\1_{u_i}-\1_{v_i}\}_{i=1}^n$ if and only if it can be
represented in the form
\begin{equation}\label{E:P}P(f)=\sum_{i=1}^n l_i(f)\frac{f_i}{\|f_i\|_\tc},\end{equation}
where:

\begin{itemize}

\item $f_i=\1_{u_i}-\1_{v_i}$

\item $l_i$ are Lipschitz functions, and
$l_i(f_j)=\delta_{i,j}\|f_j\|_\tc=\delta_{i,j}d(u_j,v_j)$
($\delta_{i,j}$ is the Kronecker delta).

\item $\|Pf\|_\tc\le \|f\|_\tc$ for every $f\in\tc(M)$ of the form
$f=\1_w-\1_z$ for $w,z\in M$.

\end{itemize}

Since $\{f_i\}_{i=1}^n$ are linearly independent and the dual of
$\tc(M)$ is the space of the Lipschitz functions on $M$, which
take value $0$ at the base point (see \cite[Theorem 10.2]{Ost13}),
any projection onto the subspace spanned by $\{f_i\}_{i=1}^n$ is
of the form \eqref{E:P} for some Lipschitz functions $\{l_i\}$
satisfying
$l_i(f_j)=\delta_{i,j}\|f_j\|_\tc=\delta_{i,j}d(u_j,v_j)$.
\medskip

It remains to show the condition $\|Pf\|_\tc\le \|f\|_\tc$ for
$f\in\tc(M)$ of the form $f=\1_w-\1_z$ implies that $\|P\|\le 1$.
This follows from our definitions and observations made above: In
fact, since for any $g\in\tc(M)$ there exists a transportation
plan of minimal cost, we can represent $g$ as a sum
$g=\sum_{i=1}^m g_i$, where $g_i$ are of the form
$g_i=b_i(\1_{w_i}-\1_{z_i})$, $b_i\in\mathbb{R}$, and
$\|g\|_\tc=\sum_{i=1}^m\|g_i\|_\tc$. Therefore we get
\[\|Pg\|_\tc=\left\|P\left(\sum_{i=1}^m g_i\right)\right\|_\tc\le\sum_{i=1}^m\|Pg_i\|_\tc\le \sum_{i=1}^m\|g_i\|_\tc=\|g\|_\tc,
\]
and thus $\|P\|\le 1$.
\medskip

Our approach to the construction of suitable functions $l_i$ is
based on the Duality Theorem of Linear Programming and the Edmonds
\cite{Edm65} algorithm for the minimum weight perfect matching
problem. We use the description of the algorithm in the form given
in \cite[Theorem 9.2.1]{LP09}, where it is shown that the minimum
weight perfect matching problem on a complete graph $G$ with even
number of vertices and weight $w:E(G)\to \mathbb{R}$, $w\ge 0$,
can be reduced to the following linear program. (An {\it odd cut}
in $G$ is the set of edges joining a subset of $V(G)$ of odd
cardinality with its complement, a {\it trivial odd cut} is a set
of edges joining one vertex with its complement. If $x$ is a
real-valued function on $E(G)$ and $A$ is a set of edges, we write
$x(A)=\sum_{e\in A}x(e)$.)
\medskip

\begin{itemize}

\item {\bf (LP1)} minimize $w^\top \cdot x$  (where $x:E(G)\to
\mathbb{R}$)
\medskip

\item subject to

\begin{itemize}

\item[(1)] $x(e)\ge 0$ for each $e\in E(G)$

\item[(2)] $x(C)=1$ for each trivial odd cut $C$

\item[(3)] $x(C)\ge 1$ for each non-trivial odd cut $C$.

\end{itemize}
\end{itemize}

We  introduce a variable $y_C$ for each odd cut $C$.
\medskip

The dual program of the program {\bf (LP1)} is:

\begin{itemize}

\item {\bf (LP2)} maximize $\sum_C y_C$

\item subject to

\begin{itemize}

\item[(D1)] $y_C\ge 0$ for each non-trivial odd cut $C$

\item[(D2)] $\sum_{C~ {\rm containing~ }e} y_C\le w(e)$ for every
$e\in E(G)$.

\end{itemize}
\end{itemize}

The Duality in Linear Programming \cite[Section 7.4]{Sch86} (see
also a summary in \cite[Chapter 7]{LP09}) states that the optima
{\bf (LP1)} and {\bf (LP2)} are equal. (In the general case we
need to require the existence of vectors satisfying the
constraints and finiteness of one of the optima.)
\medskip

This means that the total length of the minimum weight perfect
matching coincides with the sum of entries of the optimal solution
of the dual problem.
\medskip

We complete our proof of the existence of norm-$1$ projection $P$
of the desired form by proving the following two lemmas.

\begin{lemma}\label{L:IfPos} Suppose that there is an optimal dual solution satisfying
$y_C\ge 0$ for all odd cuts $C$ {\bf including trivial ones}. Then
there exist functions $l_i$ for which $P$ defined by \eqref{E:P}
is a norm-$1$ projection.
\end{lemma}

\begin{lemma}\label{L:PosExists} If the weight function $w:E(G)\to\mathbb{R}$ corresponds to a metric on $V(G)$ (this means that $w(uv)=d(u,v)$ for some
metric $d$ on $V(G)$), then there is an optimal dual solution
satisfying $y_C\ge 0$ for all odd cuts, including trivial ones.
\end{lemma}

\begin{proof}[Proof of Lemma \ref{L:IfPos}] Let $\mathcal{M}$ be the minimum weight perfect
matching, then $e\in\mathcal{M}$ is of the form $u_iv_i$. We
introduce the function $l_i:V(G)\to\mathbb{R}$ by

\begin{equation}\label{E:Def_l_i}
l_i(w)=\begin{cases} 0~~&\hbox{ if } w=u_i\\
\sum_{C~{\rm contains}~u_iv_i~{\rm and ~separates }~u_i~{\rm and}~
w} y_C ~~&\hbox{ if } w\ne u_i.
\end{cases}
\end{equation}

We claim that the function $l_i$ has the following desired
properties:

\begin{enumerate}

\item $l_i$ is $1$-Lipschitz.

\item $l_i(v_i)-l_i(u_i)=d(v_i,u_i)$.

\item $l_i(v_j)-l_i(u_j)=0$ if $j\ne i$.

\item $\sum_{i=1}^n |l_i(w)-l_i(z)|\le d(w,z)$ for every $w,z\in
M=V(G)$.

\end{enumerate}

The discussion following \eqref{E:P} implies that these conditions
imply that the obtained $P$ is a norm-$1$ projection.
\medskip

Proofs of 1--4:

\begin{enumerate}

\item $|l_i(w)-l_i(z)|\le\sum_{C~{\rm separates }~w~{\rm and}~ z}
y_C\le w(wz)=d(w,z)$, where in the first inequality we used the
definition of $l_i$, in the second we used (D2). Observe also that
item 1 follows from the stronger inequality in item 4, which we
prove below.

\item $l_i(v_i)-l_i(u_i)=d(v_i,u_i)$.

The corresponding argument is shown in \cite[p.~371]{LP09}. We
reproduce it. We have

\begin{equation}\label{E:ForOptDual} w(\mathcal{M})=\sum_{e\in\mathcal{M}} w(e)\ge \sum_{e\in
\mathcal{M}}~~\sum_{C~ {\rm containing~ }e}
y_C=\sum_{C}|\mathcal{M}\cap C|y_C\ge\sum_C y_C,\end{equation}
where in the first inequality we used (D2) and in the second
inequality we used $|\mathcal{M}\cap C|\ge 1$ for each odd cut.

If $y_C$ is an optimal dual solution, we get that the leftmost and
the rightmost sides in \eqref{E:ForOptDual} coincide, and
therefore \begin{equation}\label{E:WeighAtt} w(e)=\sum_{C~ {\rm
containing~ }e}y_C\end{equation} for each $e\in\mathcal{M}$ and

\begin{equation}\label{E:Int1} |\mathcal{M}\cap C|=1 \hbox{ for each non-trivial odd cut }C
\hbox{ satisfying }y_C>0\end{equation}
\medskip

Equality \eqref{E:WeighAtt} implies $l_i(v_i)-l_i(u_i)=\sum_{C~
{\rm containing~ }u_iv_i}y_C-0=w(u_iv_i)=d(u_i,v_i)$.

\item $l_i(v_j)-l_i(u_j)=0$ if $j\ne i$.

This equality follows from \eqref{E:Int1}. In fact, equality
\eqref{E:Int1} implies that none of the cuts with $y_C>0$
containing $u_iv_i$ can contain $u_jv_j$ for $j\ne i$, and thus
$l_i(v_j)= l_i(u_j)$ for all $j\ne i$.
\medskip

\item $\sum_{i=1}^n |l_i(w)-l_i(z)|\le d(w,z)$ for every $w,z\in
M$.

To prove this inequality we observe that $|l_i(w)-l_i(z)|\le
\sum_{C\in S_i(w,z)} y_C$, where $S_i(w,z)$ is the set of cuts $C$
with $y_C>0$ which simultaneously separate $u_i$ from $v_i$ and
$w$ from $z$. It is important to observe that \eqref{E:Int1}
implies that the sets $\{S_i(w,z)\}_{i=1}^n$ are disjoint.
Therefore, by (D2), $\sum_{i=1}^n |l_i(w)-l_i(z)|\le d(w,z)$.
\end{enumerate}
\end{proof}

\begin{proof}[Proof of Lemma \ref{L:PosExists}] We follow the
presentation in \cite[Section 9.2]{LP09} of the Edmonds algorithm
for construction of an optimal dual solution. To prove the lemma
it suffices to show that the assumption that $w$ corresponds to a
metric implies that when we run the algorithm we maintain $y_C\ge
0$ in each step, even for trivial odd cuts.

We decided not to copy the whole Section 9.2 at a price that we
expect readers (who do not remember the algorithm) to have
\cite[Section 9.2]{LP09} handy.

The beginning of the algorithm can be described as follows: we
assign the number $y_C=\frac12\min_{u,v}d(u,v)$ to all trivial
cuts $C$ and set $y_C=0$ for all nontrivial cuts $C$. This
function on the set of all odd cuts satisfies the conditions (D1)
and (D2). Such functions are called {\it dual solutions}. For a
dual solution $y$ we form a graph $G_y$ whose vertex set is $V(G)$
and edge set is defined by
\[E_y=\left\{e\in E(G):\sum_{C~ {\rm containing~ }e}
y_C=w(e)\right\}.\]

It is clear that with $y_C$ defined as above we get a graph $G_y$
which can contain any number of edges between $1$ and
$\frac{n(n-1)}2$.

In each step of the Edmonds algorithm we construct not only the
function $y_C$, but also a set $\mathcal{H}$ of odd cardinality
subsets of $V(G)$ satisfying four conditions listed in
\cite[(P-1)--(P-4), page 372]{LP09}. We list only the first two
conditions, because the contents of the last two conditions does
not affect our modification of the argument in \cite[Section
9.2]{LP09}.
\medskip

(P-1) $\mh$ is nested, that is, if $S,T\in\mh$, then either
$S\subset T$ or $T\subset S$ or $S\cap T=\emptyset$.
\medskip

(P-2) $\mh$ contains all singletons of $V(G)$.
\medskip

At the end of the first step described above the set $\mh$ is let
to be the set of singletons (and all of the desired conditions are
satisfied).
\medskip

After that the following step is repeated and the function $y_C$
is modified till the graph $G_y'$ (described below) becomes a
graph having perfect matching.

Let $S_1,\dots,S_k$ be the (inclusionwise) maximal members of
$\mh$. It follows from (P-1) that $S_1,\dots,S_k$ are mutually
disjoint and from (P-2) that they form a partition of $V(G)$. Let
$G_y'$ denote the graph obtained from $G_y$ by contracting each
$S_i$ to a single vertex $s_i$. Since $|V(G)|$ is even, but $S_j$
is odd, it follows that $k:=|V(G_y')|$ is even.

Suppose that $G_y'$ does not have a perfect matching. Let
$A(G_y')$, $C(G_y')$, and $D(G_y')$ be the sets of the
Gallai-Edmonds decomposition for $G_y'$ (see \cite[Section
3.2]{LP09}).

We use the notation $A(G_y')=\{s_1,\dots,s_m\}$ and denote the
components of the subgraph of $G_y'$ induced by $D(G_y')$ by
$H_1,\dots,H_{m+d}$, where $d$ is the number of vertices which are
not matched in a maximum matching in $G_y'$. Let
\[T_i=\bigcup_{s_j\in V(H_i)}S_j.\]

Now we modify the dual solution $y$ as follows (by $\nabla(S)$ we
denote the set of edges connecting a vertex set $S$ with its
complement):

\[y^t_{\nabla (S_j)}=y_{\nabla (S_j)}-t\quad (1\le j\le m),\]
\[y^t_{\nabla (T_i)}=y_{\nabla (S_i)}+t\quad (1\le i\le m+d),\]
\[y^t_C=y_C, \quad \hbox{otherwise}.\]

In this formula $t$ is chosen as the minimum of three numbers,
$t_1,t_2,t_3$, defined as:

\[t_1=\min\{y_{\nabla(S_j)}:~1\le j\le m, |S_j|>1\},\]
\[t_2=\min\{w(e)-\sum_{e\in
C}y_C:~e\in\nabla(T_1)\cup\dots\cup\nabla(T_{m+d})\backslash(\nabla(S_1)\cup\dots\cup\nabla(S_m))\},\]
\[t_3=\frac12\min\{w(e)-\sum_{e\in
C}y_C:~e\in(\nabla(T_i)\cap\nabla(T_j)),~1\le i<j\le m+d\}.\]

It is clear from the definition of $t_1$ that negative
coefficients can appear only for those $S_j$ which are singletons.
So suppose that $S_j$ is a singleton, $S_j=\{v\}$. To complete the
proof of Lemma \ref{L:PosExists} it remains to show that $t_3 \le
y_{\nabla(v)}$, and so $y^t_{\nabla(v)}$ is still nonnegative.

Because of the positive surplus condition in \cite[Theorem 3.2.1
(c)]{LP09}, the vertex $v$ is connected in $G_y$ with at least two
of the sets $\{T_i\}_{i=1}^{m+d}$, suppose that these are sets
$T_{i_1}$ and $T_{i_2}$. Let $u\in T_{i_1}$ and $w\in T_{i_2}$ be
adjacent to $v$ in $G_y$. Let $\{U_p\}_{p=1}^\tau$ be the elements
of $\mh$ containing $u$ and let $\{W_q\}_{q=1}^\sigma$ be the
elements of $\mh$ containing $w$. Since the edges $uv$ and $wv$
are in $G_y$, we have
\begin{equation}\label{E:uv}w(uv)=y_{\nabla(v)}+\sum_{p=1}^\tau
y_{\nabla U_p}\end{equation}
\begin{equation}\label{E:wv}
w(wv)=y_{\nabla(v)}+\sum_{q=1}^\sigma y_{\nabla W_q}\end{equation}
On the other hand, the definition of $t_3$ and our choice of
$S_1,\dots,S_k$ imply that
\[\begin{split}t_3&\le\frac12\left(w(uw)-\sum_{p=1}^\tau y_{\nabla U_p}-\sum_{q=1}^\sigma y_{\nabla
W_q}\right)\\&\le\frac12\left(\left(w(uv)-\sum_{p=1}^\tau
y_{\nabla U_p}\right)+\left(w(vw)-\sum_{q=1}^\sigma y_{\nabla
W_q}\right)\right)\\&=y_{\nabla(v)},\end{split}\] where in the
second inequality we use the triangle inequality for the distance
corresponding to weight $w$, and in the last equality we use
\eqref{E:uv} and \eqref{E:wv}.
\end{proof}

\section{Isometric copies of $\ell_\infty^n$ in $\tc(M)$}
\label{S:ellinftyn}

As is well-known the spaces $\{\ell_\infty^n\}$ admit
low-distortion and even isometric embeddings into some
transportation cost spaces. This follows from the basic property
of $\tc(M)$: it contains an isometric copy of $M$ (see
\eqref{E:EmbBasePt}).
\medskip

Another related fact is the following immediate consequence of the
Bourgain discretization theorem (see \cite{Bou87}, \cite{GNS12},
\cite[Section 9.2]{Ost13}): for sufficiently large $m$ the
transportation cost space on the set of integer points in
$\ell_\infty^n$ with absolute values of coordinates $\le m$
contains an almost-isometric copy of $\ell_\infty^n$.
\medskip

In the next example we need the following well-known fact (see
\cite[Section 3.3]{Wea18}, \cite[Section 1.6]{OO19}): If $(M,d)$
is a complete metric space, then $\tc(M)$ contains the vector
space of differences between finite positive compactly supported
measures $\mu$ and $\nu$ on $M$ with the same total masses and
$\|\mu-\nu\|_\tc$ is equal to the quantity
$\mathcal{T}_1(\mu,\nu)$ defined in the following way.

A {\it coupling} of a pair of finite positive Borel measures
$(\mu,\nu)$ with the same total mass on $M$ is a Borel measure
$\pi$ on $M\times M$ such that $\mu(A)=\pi(A\times M)$ and
$\nu(A)=\pi(M\times A)$ for every Borel measurable $A\subset M$.
The set of couplings of $(\mu,\nu)$ is denoted $\Pi(\mu,\nu)$. We
define
$$
\mathcal{T}_1(\mu,\nu):= \inf_{\pi\in
\Pi(\mu,\nu)}\bigg(\iint_{M\times M} d(x,y)\,d\pi(x,y)\bigg).
$$

The result of Godefroy and Kalton \cite[Theorem 3.1]{GK03} has the
following special case:

\begin{example} Let us consider the following (non-discrete)
transportation problems on the unit cube $[0,1]^n$ with its
$\ell_\infty$-distance:

$P_i$: ``available'' is the Lebesgue measure on the face $x_i=0$,
``needed'' is the Lebesgue measure on the face $x_i=1$.
\end{example}

It is clear that $P_i$ has cost $1$, and actually any
measure-preserving transportation from bottom to top does the job.
The easiest transportation plan is to move each point from the the
face $x_i=0$ to the point with the same coordinates, changing only
$x_i$ from $0$ to $1$.
\medskip

It is not that easy to see that $\sum_{i=1}^n\theta_iP_i$  has
cost $1$. This can be done as follows. By symmetry it suffices to
consider the case where all $\theta_i=1$. In this case we move
each point from the surface with ``availability'' to the surface
with ``need'' in the direction of the diagonal $(1,\dots,1)$. It
is easy to see that it will be a bijection between points of
``availability'' and ``need''. The cost can be computed as the
following integral:

\[\begin{split}n\int_0^1 t(-d(1-t)^{n-1})&=n(n-1)\int_0^1 t(1-t)^{n-2}dt\\&=
n(n-1)\int_0^1((1-t)^{n-2}-(1-t)^{n-1})dt\\&
=n(n-1)\left.\left(\frac{(1-t)^n}n-\frac{(1-t)^{n-1}}{n-1}\right)\right|_0^1=1.
\end{split}\]

We are interested in constructing finite metric spaces $M$ for
which $\tc(M)$ contains $\ell_\infty^n$ isometrically. So far we
succeeded to do this only for $n=3$ and $n=4$ (the case $n=2$ is
easy, see Observation \ref{O:Dim2}).

\begin{example}[Finite $M$ with $\tc(M)$ containing
$\ell_\infty^3$ isometrically] The set $M$ which we consider is a
subset of the surface of the cube $[0,1]^3$ endowed with its
$\ell_\infty$ distance. Transportation problem $P_i$ is described
in the following way: ``available'' is $\frac16$ at each midpoint
of the edge in the face $x_i=0$ and $\frac13$ at the center of the
face; ``needed'' is at the similar points with $x_i=1$.
\end{example}

The transportation cost for $P_i$ is $1$ - just shift from $x_i=0$
to $x_i=1$. Again by symmetry it suffices to show that the cost of
$P_1+P_2+P_3$ is $1$.\medskip

Consider faces with $x_i=0$ as colored ``red'' and faces with
$x_i=1$ as colored ``blue''. It is clear that availability and
need on two-dimensional faces which are on the boundary between
blue and red cancel each other. There will be $6$ points of
availability left. Three of them are on edges, and three  are
centers of faces. The value is $\frac13$ at each. So to achieve
cost $1$ it suffices to match red and blue vertices in such a way
that the distance between any two matched vertices is $\frac12$.

This is possible. To achieve this we match red points which are
centers of edges with blue vertices which are centers of faces and
red points which are centers of faces with blue vertices which are
centers of edges.\medskip

Our example in dimension $4$ is even more symmetric.
\medskip

\begin{example}[Finite $M$ with $\tc(M)$ containing
$\ell_\infty^4$ isometrically] The set $M$ which we consider is a
subset of the surface of the cube $[0,1]^4$ endowed with its
$\ell_\infty$ distance. Transportation problem  $P_i$ is described
in the following way: ``available'' is $\frac16$ at the center of
each of each $2$-dimensional face of the face $x_i=0$; ``needed''
is at the similar points with $x_i=1$.
\end{example}

The transportation cost for $P_i$ is $1$ - just shift from $x_i=0$
to $x_i=1$. Again by symmetry it suffices to show that the cost of
$P_1+P_2+P_3+P_4$ is $1$.\medskip

As in the above discussion with blue and red we see that half of
the availability and need will cancel each other.
\medskip

The remaining availability of value $\frac13$ will be concentrated
at $6$ centers of 2-dimensional faces of 3-dimensional faces. Each
of these centers will have coordinates $\frac12,\frac12,1,1$ in
some order. The need of value $\frac13$ will concentrated at $6$
points with coordinates $\frac12,\frac12,0,0$. Cancellation will
occur at points with coordinates $\frac12,\frac12,0,1$. \medskip

To get the transportation plan of cost $1$ we need to find a
matching between points with coordinates  $\frac12,\frac12,1,1$
and points with coordinates $\frac12,\frac12,0,0$, such that the
distance between each pair of matched vertices is $\frac12$. Such
matching is obvious.
\medskip

\section*{Acknowledgement}

The contribution of the first two authors is the outcome of their
research done under the supervision of the third author. The
authors gratefully acknowledge the support by the National Science
Foundation grant NSF DMS-1700176. The first author was supported
by GAAP. The authors thank the referee for corrections and
suggestions which helped to improve the presentation. The
third-named author also would like to thank Florin Catrina and
Beata Randrianantoanina for useful discussions. During the final
stage of the work on this paper the third author was on research
leave supported by St. John's University.

\begin{small}

\textsc{Department of Mathematics and Computer Science, St. John's
University, 8000 Utopia Parkway, Queens, NY 11439, USA} \par
  \textit{E-mail address}: \texttt{seychelle.khan16@my.stjohns.edu} \par
\medskip

\textsc{Department of Mathematics and Computer Science, St. John's
University, 8000 Utopia Parkway, Queens, NY 11439, USA} \par
  \textit{E-mail address}: \texttt{mutasim.mim16@my.stjohns.edu} \par
\medskip

\textsc{Department of Mathematics and Computer Science, St. John's
University, 8000 Utopia Parkway, Queens, NY 11439, USA} \par
  \textit{E-mail address}: \texttt{ostrovsm@stjohns.edu} \par

\end{small}

\end{large}
\end{document}